\documentclass{amsart}
\usepackage{amsmath,amsthm,amssymb,amscd}

 \newtheorem{thm}{Theorem}[section]
 \newtheorem{cor}[thm]{Corollary}
 \newtheorem{lem}[thm]{Lemma}
 
 \theoremstyle{definition}
 \newtheorem{defn}[thm]{Definition}
 \theoremstyle{remark}
 \newtheorem{rem}[thm]{Remark}
 \newtheorem{ex}[thm]{Example}
 \numberwithin{equation}{section}

\newcommand{\bb}[1]{\mbox{$\mathbb{#1}$}}

\newcommand{\bR}{{\mathbb R}}
\newcommand{\bC}{{\mathbb C}}
\newcommand{\bP}{{\mathbb P}}

\newcommand{\bZ}{{\mathbb Z}}

\newcommand{\bQ}{{\mathbb Q}}

\newcommand{\cO}{{\mathcal O}}
\newcommand{\cD}{{\mathcal D}}
\newcommand{\cF}{{\mathcal F}}
\newcommand{\cG}{{\mathcal G}}

\begin{document}
%-------------------------------------------------------------------------
% editorial commands: to be inserted by the editorial office
%
%\firstpage{1}
%\volume{228}
%\Copyrightyear{2004}
%\DOI{003-0001}
%
%
%\seriesextra{Just an add-on}
%\seriesextraline{This is the Concrete Title of this Book\br H.E. R and S.T.C. W, Eds.}
%
% for journals:
%
%\firstpage{1}
%\issuenumber{1}
%\Volumeandyear{1 (2004)}
%\Copyrightyear{2004}
%\DOI{003-xxxx-y}
%\Signet
%\commby{inhouse}
%\submitted{March 14, 2003}
%\received{March 16, 2000}
%\revised{June 1, 2000}
%\accepted{July 22, 2000}
%
%
%
%---------------------------------------------------------------------------
%Insert here the title, affiliations and abstract:
%
\title[Chern classes and transversality for singular spaces]
 {Chern classes and transversality for\\ singular spaces}
%----------Author 1
\author[J. Sch\"urmann ]{J\"org Sch\"urmann}

\address{%
Mathematische Institut\\
Universit\"at M\"unster\\
Einsteinstr. 62\\
48149 M\"unster\\
Germany}

\email{jschuerm@uni-muenster.de}

\begin{abstract}
 In this paper we compare different notions of transversality for possible singular complex algebraic or analytic subsets of an 
 ambient complex manifold and prove a refined intersection formula for their Chern-Schwartz-MacPherson classes. In case of
 a transversal intersection of complex Whitney stratified sets, this result is well known. For splayed subsets it was conjectured 
 (and proven in some cases) by Aluffi and Faber. Both notions are stronger than a micro-local ``non-characteristic intersection''
 condition for the characteristic cycles of (associated) constructible functions, which nevertheless is enough to imply the asked 
 refined intersection formula for the Chern-Schwartz-MacPherson classes. The proof is based the multiplicativity of
 Chern-Schwartz-MacPherson classes with respect to cross products, as well as a new Verdier-Riemann-Roch theorem
 for ``non-characteristic pullbacks''.
\end{abstract}

%\thanks{}

%----------classification, keywords, date


%\date{January 1, 2004}
%----------additions
\dedicatory{Dedicated to Pepe Seade on his $60$th birthday}
%%% ----------------------------------------------------------------------

\maketitle

\section{Introduction}
In this paper we work in the embedded complex analytic or algebraic
context, with $X$ and $Y$  closed (maybe singular) subspaces in the ambient complex manifold $M$.
And we want to show under suitable ``transversality assumptions'' the following refined intersection formula for 
their Chern-Schwartz-MacPherson classes:
\begin{equation}\label{intro1}
   d^!\left( c_*(X) \times c_*(Y) \right) = c(TM)\cap c_*(X\cap Y) \in H_*(X\cap Y)\:.
  \end{equation}
Here $H_{*}(X)$ denotes either the Borel-Moore homology group   $H^{BM}_{2*}(X,\bZ)$ in even degrees
or in the algebraic context the Chow group $CH_{*}(X)$, with 
$$d^!: H_*(X\times Y) \to H_*(X\cap Y)$$
the corresponding refined pullback for the (regular) diagonal embedding $d: M\to M\times M$ of the ambient complex manifold $M$
(as recalled in the next section). Note that for $X$ compact, the topological Euler characteristic $\chi(X)$ of $X$ is given by
$$\chi(X)=deg(c_*(X))=deg(c_0(X))\:.$$
So formula (\ref{intro1}) shows that in general for $X$ and $Y$ compact the Euler characteristic $\chi(X\cap Y)$
of the intersection cannot be given just in terms of $\chi(X)$ and $\chi(Y)$, but that the information of their total
Chern-Schwartz-MacPherson classes $c_*(X)$ and $c_*(Y)$ is needed.\\

For $X$ and $Y$ smooth complex submanifolds, all these different notions of ``trans\-versality'' for singular subspaces
just reduce to the classical notion of trans\-versality, so that $X\cap Y$ also becomes a smooth complex  submanifold of $M$,
with normal bundle 
$$N_{X\cap Y}M = N_XM|_{X\cap Y}\oplus N_YM|_{X\cap Y} \:.$$
And then (\ref{intro1}) easily follows from the fact, that
$$c_*(Z)=c(TZ)\cap [Z] = c(N_ZM)^{-1}\cap \left( c(TM)\cap [Z] \right)$$
for $Z$ a closed smooth complex submanifold of $M$ (with $c$ the total Chern class). 
Also recall that the smooth complex submanifolds $X$ and $Y$ of $M$ 
intersect trans\-versally, iff the diagonal embedding $d: M\to M\times M$ is transversal to $X\times Y$, with $d^{-1}(X\times Y)=X\cap Y$
and 
$$N_{X\cap Y}M = d^*(N_{X\times Y}(M\times M)) = d^*(N_XM\times N_YM)\:.$$
And this last viewpoint can be generalized in different ways to singular complex subspaces. \\

Maybe the best known notion of tranversality for singular $X$ and $Y$ is the transversality as complex Whitney stratified subsets,
i.e. both are endowed with complex Whitney b-regular stratifications such that all strata $S$ of $X$ and $S'$ of $Y$ are transversal.
Equivalently, the diagonal embedding $d$ is transversal to all strata $S\times S'$ of the induced product Whitney stratification of $X\times Y$.
And then the intersection formula
(\ref{intro1}) is well known, see e.g. \cite{CMS}[Thm.3.3] or \cite{Sch4}[Cor.0.1 and the discussion afterwords].
Another notion of transversality for singular complex subspaces $X$ and $Y$ in the ambient complex manifold $M$ was studied  by
Aluffi and Faber \cite{AF, AF2} (and first  introduced and characterized  by Faber \cite{Fa} in the hypersurface case):
\begin{defn}\label{splayed-sets}
$X$ and $Y$ are \emph{splayed at a point} $p\in M$, if there is near $p$ a local analytic isomorphism $M=V_1\times V_2$
 of analytic manifolds so that $X$ resp. $Y$ can be defined by an ideal in the coordinates of $V_1$ resp. $V_2$.
 $X$ and  $Y$
 are \emph{splayed} if they are splayed at all points $p\in X\cap Y$.
\end{defn}
Note that also in the complex algebraic context, these local coordinates are only asked for in the local analytic context.
And Aluffi and Faber \cite{AF, AF2} conjectured (and could prove in some cases) the intersection formula
(\ref{intro1}) for $X$ and $Y$ splayed. 
The problem is of course, that both notions of transversality cannot be directly compared and are of very different type,
with ``stratified transversality'' more of geometric and ``splayedness'' more of algebraic nature.\\

We will gereralize in the next section both notions even to constructible functions $\alpha\in F(X) $ and $\beta\in F(Y)$,
showing that both are stronger than the  micro-local ``non-characteristic intersection'' condition, that the diagonal embedding
$d: M\to M\times M$ is ``non-characteristic'' with respect to the support 
$$supp(CC(\alpha\times \beta))\subset T^*(M\times M)$$
of the characteristic cycle 
of $\alpha\times \beta$. Nevertheless this micro-local ``non-characteristic intersection'' condition implies the following generalization
of the refined intersection formula  (\ref{intro1}) 
even for the Chern-Schwartz-MacPherson classes of constructible functions:

\begin{thm}\label{main-intro}
  Let $X, Y$ be two closed subspaces of the complex (algebraic) manifold $M$ with given constructible functions $\alpha\in F(X)$ and  
  $\beta \in F(Y)$. Assume that the diagonal embedding $d: M\to M\times M$ is non-characteristic with respect to $supp(CC(\alpha\times \beta))$ 
  (e.g. $\alpha$ and $\beta$ are splayed or stratified transversal).
  Then 
  \begin{equation}\label{eq-main-intro}
   d^!\left( c_*(\alpha) \times c_*(\beta) \right) = c(TM)\cap c_*(\alpha \cdot \beta) \in H_*(X\cap Y)\:.
  \end{equation}
  In particular
  \begin{equation}
    c_*(\alpha) \cdot  c_*(\beta)  = c(TM)\cap c_*(\alpha \cdot \beta) \in H_*(M)\:.
  \end{equation}
\end{thm}
 
By definition of the MacPherson Chern class in \cite{MP}, $c_*(Z):=c_*(1_Z)$ for $Z=X,Y$ or $X\cap Y$, so that (\ref{eq-main-intro}) implies the formula (\ref{intro1}) by $1_X\cdot 1_Y=1_{X\cap Y}$. 
Let us mention here, that
Brasselet and Schwartz \cite{BS} (see also \cite{AB})
showed that the MacPherson's Chern class $c_*(1_X)$ corresponds to the Schwartz class $c^S(X) \in H^{2*}_X(M)$ (see \cite{Schw1, Schwa2}) by  Alexander duality for $X$ embedded in the smooth complex manifold $M$. That is why the total homology class $c_*(X)= c_*(1_X)$ is called the Chern--Schwartz--MacPherson class of $X$.

Other natural Chern classes of a singular complex algebraic or analytic set $Z$ are the 
Aluffi-Chern class $c^A_*(Z):=c_*(\nu_Z)$ defined by the constructible Behrend function $\nu_X$ introduced in \cite{B},
or for $Z$ pure-dimensional  the Mather-Chern class $c^M_*(Z):=c_*(Eu_Z)$ defined by the famous constructible Euler obstruction
function $Eu_Z$ of MacPherson \cite{MP}. Both functions $\nu_Z$ and $Eu_Z$ commute with cross-products, restriction to open subsets and switching from the algebraic to the  analytic context, with $Eu_Z=1_Z=(-1)^{dim(Z)}\cdot \nu_Z$ for $Z$ smooth.

\begin{cor} Let $X,Y$ be two closed splayed subspaces of the complex manifold $M$, with $m=dim(M)$.
Then also $\nu_X$ and $\nu_Y$ are splayed, with $\nu_X\cdot \nu_Y=(-1)^{m}\cdot \nu_{X\cap Y}$ so that
\begin{equation}
   d^!\left( c^A_*(X) \times c^A_*(Y) \right) = (-1)^{m}\cdot c(TM)\cap c^A_*(X\cap Y) \in H_*(X\cap Y)\:.
  \end{equation}
If in addition $X$ and $Y$ are pure-dimensinal, then also $Eu_X$ and $Eu_Y$ are splayed, with
$X\cap Y$ pure-dimensional and $Eu_X\cdot Eu_Y=Eu_{X\cap Y}$ so that
\begin{equation}
   d^!\left( c^M_*(X) \times c^M_*(Y) \right) =  c(TM)\cap c^M_*(X\cap Y) \in H_*(X\cap Y)\:.
  \end{equation}
\end{cor}
 
 The proof of theorem \ref{main-intro}  is based on the multiplicativity of
 Chern-Schwartz-MacPherson classes for  cross products  (see \cite{K, KYo}):
$$ c_*(\alpha \times \beta)= c_*(\alpha)\times c_*(\beta) \:,$$
as well as the following  new Verdier-Riemann-Roch theorem
 for ``non-characteristic pullbacks'' (applied to the diagonal embedding $d: M\to M\times M$):
 
 \begin{thm}\label{VRR-int}
Let $f: M\to N$ be a morphism of complex (algebraic) manifolds, with $Y\subset N$ a closed subspace and
$X:=f^{-1}(Y)\subset M$. Assume that $\gamma\in F(Y)$ is a constructible function such that $f$ is
\emph{non-characteristic} with respect to the support $supp(CC(\gamma))\subset T^*N|Y\subset T^*N$ of the characteristic cycle
$CC(\gamma)$ of $\gamma$. Then
\begin{equation}
 f^!(c(TN)^{-1} \cap c_*(\gamma)) = c(TM)^{-1} \cap c_*(f^*(\gamma)) \in H_*(X)\:,
\end{equation}
with $f^!: H_*(Y)\to H_*(X)$ the Gysin map induced by the morphism $f: M\to N$ of complex (algebraic) manifolds.
\end{thm}

This Verdier-Riemann-Roch theorem for ``non-characteristic pullbacks'' is the main result of this paper,
from which the refined intersection formula (\ref{eq-main-intro}) directly follows in the spirit of Lefschetz' definition of intersection theory 
via cross-products and (refined) pullbacks for the diagonal map (as explained in the next section).
The proof of Theorem \ref{VRR-int} is based on the micro-local approach to Chern-Schwartz-MacPherson classes via characteristic cycles 
of constructible functions (as recalled in the last section, and see e.g. \cite{Gi1, Gi2, Ken1, Sab, Sch-lect}).
Note that the micro-local notion of ``non-characteristic'' has its origin in the theory of (holonomic) ${\mathcal D}$-modules
(as in \cite{Gi1, Gi2}) as well as in the micro-local sheaf theory of Kashiwara-Schapira \cite{KS}.
Nevertheless, in our context we think of it as a micro-local ``transversality condition'' fitting nicely with the geometry of
Chern-Schwartz-MacPherson classes of constructible functions. 
As a byproduct of our proof we  also get the following
``micro-local intersection formula'':

\begin{cor}\label{cor-intro}
Let $M$ a complex (algebraic) manifold of dimension $m=dim(M)$, with $\alpha, \beta\in F(M)$ given constructible functions.
Assume that the diagonal embedding $d: M\to M\times M$ is non-characteristic with respect
to $supp(CC(\alpha\times \beta))$ (e.g. $\alpha, \beta\in F(M)$ are splayed or stratified transversal), with  $supp(\alpha\cdot \beta)$ compact.

Then also  $supp(CC(\alpha)\cap CC(\beta))\subset T^*M$ is compact, with
\begin{equation}\label{mic-intro}
\chi(M;\alpha\cdot \beta) = (-1)^m\cdot deg( CC(\alpha)\cap CC(\beta))\:.
\end{equation} 
\end{cor}

In the end of the  next section we also illustrate some other situations,
where such a ``non-characteristic condition'' follows from suitable ``splayedness'' assumptions. 
Characteristic cycles are of ``cotangential nature'', since they are living in the cotangent bundle $T^*M$ of the ambient manifold $M$.
So for the pullback situation of a holomorphic map $f: M\to N$ of complex manifolds we have to study the associated correspondence
$$T^*M \leftarrow f^*TN \rightarrow T^*N$$
of cotangent bundles. And here the ``non-characteristic'' condition shows automatically up for the definition of the pullback of a 
characteristic cycle living in $T^*N$. Let us finally point out, that there are also other \emph{intrinsic}  notions of 
Fulton- and Fulton-Johnson-Chern-classes (see \cite{Fu}[Example 4.2.6])
$$c_*^F(X)=c^*(TM)\cap s_*(C_XM) \, , \, c_*^{FJ}(X)=c^*(TM)\cap   s_*({\mathcal N}_XM)\in H_*(X)$$
of a singular complex variety $X\subset M$ embedded into a complex manifold $M$, which are more of ``tangential nature'' defined via Segre-classes
of the normal cone $C_XM$ resp. the cone associated to the conormal sheaf ${\mathcal N}_XM$ of $X$ in $M$.
And also for them the refined intersection formula (\ref{intro1}) holds under the assumption that $X$ and $Y$ are \emph{splayed} in $M$.
For the Fulton-Chern-classes this was shown by Aluffi and Faber \cite{AF2}[Thm.III].
And in another paper we will give a different proof of this result, very close to the ideas of this paper, which will also 
apply to the Fulton-Johnson-Chern-classes. 

Finally let us point out, that results similar to  Theorem \ref{main-intro} and (\ref{intro1})
are also true for the \emph{Hirzebruch class transformation} $T_{y*}$ of \cite{BSY1, Sch-MSRI}
in the context of complex algebraic mixed Hodge modules, if one asks the ``non-characteristic property'' for the characteristic variety of the underlying (filtered) $\cD$-modules. But in this case the proof is different and follows the lines of \cite{Sch4}
(as explained elsewhere).

\section{Splayed constructible functions}
In this chapter we work in the embedded complex analytic or algebraic
context, with $X$ a closed subspace in the complex manifold $M$.
Let $F(X)=F_{alg}(X)$ or $F(X)=F_{an}(X)$ be the group of $\bZ$-valued constructible function, so that
a constructible function $\alpha\in F(X)$ in the complex algebraic (resp. analytic) context
is a (locally finite) linear combination of indicator functions $1_Z$ with $Z\subset X$ a
closed irreducible subspace. Viewing an algebraic variety as an analytic variety, one gets a canonical
injection $F_{alg}(X)\hookrightarrow F_{an}(X)$.\\

First we extend the definition of \emph{splayedness} form subspaces to constructible functions:

\begin{defn}\label{splayed-fct}
 Let $X, Y$ be two closed subspaces of $M$ with given constructible functions $\alpha\in F(X)$ and  $\beta \in F(Y)$.
 Then $\alpha$ and $\beta$ are \emph{splayed at a point} $p\in M$, if there is near $p$ a local analytic isomorphism $M=V_1\times V_2$
 of analytic manifolds so that $\alpha=\pi_1^*(\alpha')$ and $\beta=\pi_2^*(\beta')$ for some $\alpha'\in F(V_1)$ and 
 $\beta'\in F(V_2)$, with $\pi_i: V_1\times V_2\to V_i$ the projection ($i=1,2$). \\
 $\alpha\in F(X)$ and  $\beta \in F(Y)$
 are \emph{splayed} if they are splayed at all points $p\in X\cap Y$.
\end{defn}

So two algebraically constructible functions $\alpha,\beta$ are by definition splayed (at a point $p$), if this is the
case for them viewed as analytically constructible functions. Let us give some examples:

\begin{ex} 
\begin{enumerate}
 \item If $\alpha \in F(M)$ is \emph{locally constant}, then $\alpha$ and $\beta$ are splayed for any $\beta \in F(Y)$.
 Just take a local isomorphism $U=\{pt\}\times U$ with $\alpha$ constant on $U$ so that $\alpha=\alpha'\cdot 1_U=\pi_1^*(\alpha')$
 for $\alpha'\in \bZ=F(pt)$, with $\pi_1: \{pt\}\times U\to \{pt\}$ the projection.
 
 So if one wants to show that two constructible functions $\alpha\in F(X)$ and  $\beta \in F(Y)$ are splayed,
 then one only needs to check this at all points $p\in M$ were $\alpha$ and $\beta$ are not (locally) constant near $p$.
 In particular the choice of the closed subspaces $X$ and $Y$ in Definition \ref{splayed-fct} doesn't matter.
 \item Let $X, Y$ be two closed subspaces of $M$. Then $\alpha=1_X$ and $\beta=1_Y$ are splayed as constructible funstions,
 if and only if $X$ and $Y$ are splayed as closed subspaces. Moreover, in this case also the constructible indicator functions
 of their complements $1_{M\backslash X}, 1_{M\backslash Y}\in F(M)$ are splayed.
\end{enumerate}
\end{ex}

For $X$ a closed subspace in the complex manifold $M$, let $c_*: F(X)\to H_*(X)$ be the MacPherson Chern class
transformation for constructible functions, were $H_{*}(X)$ denotes either the Borel-Moore homology group in even degrees $H^{BM}_{2*}(X,\bZ)$
or in the algebraic context the Chow group $CH_{*}(X)$. In the next chapter we will explain a possible definition of $c_*$ in this embedded context
$X\subset M$ via the theory of conic Lagrangian cycles in the cotangent bundle $T^*M|X$. This implies in the algebraic context directly the commutativity 
of the following diagram, with $cl$ the cycle map (as in \cite{Fu}[Chapter 19]):
\begin{equation}
 \begin{CD}
  F_{alg}(X) @>>>  F_{an}(X)\\
  @V c_* VV @VV c_* V \\
  CH_{*}(X) @>> cl > H^{BM}_{2*}(X,\bZ)\:.
 \end{CD}
\end{equation}

One of the main results of this paper is the following
\begin{thm}\label{main-th}
  Let $X, Y$ be two closed subspaces of the complex (algebraic) manifold $M$ with given \emph{splayed} constructible functions $\alpha\in F(X)$ and  $\beta \in F(Y)$.
  Then 
  \begin{equation}
   d^!\left( c_*(\alpha) \times c_*(\beta) \right) = c(TM)\cap c_*(\alpha \cdot \beta) \in H_*(X\cap Y)\:.
  \end{equation}
  In particular
  \begin{equation}
    c_*(\alpha) \cdot  c_*(\beta)  = c(TM)\cap c_*(\alpha \cdot \beta) \in H_*(M)\:.
  \end{equation}
\end{thm}

In the algebraic context $$d^!: CH_{*}(X\times Y)\to CH_*(X\cap Y)$$ is the \emph{refined Gysin map} (\cite{Fu}[Sec.6.2]) 
associated to the regular diagonal embedding
$d: M\to M\times M$. In the analytic context it is the \emph{refined pullback} $$d^!: H^{BM}_{2*}(X\times Y,\bZ)\to H^{BM}_{2*}(X\cap Y,\bZ)$$ which under 
Poincar\'{e} duality corresponds to the pullback $$d^*: H^{2*}_{X\times Y}(M\times M,\bZ)\to H^{2*}_{X\cap Y }(M,\bZ)$$ in cohomology with support.

\begin{cor}\label{cor:spaces}
   Let $X, Y$ be two \emph{splayed} closed subspaces of the complex (algebraic)  manifold $M$. Then
  \begin{equation}
   d^!\left( c_*(X) \times c_*(Y) \right) = c(TM)\cap c_*(X\cap Y) \in H_*(X\cap Y)\:,
  \end{equation}
  in particular
  \begin{equation}\label{cor:spaces2}
    c_*(X) \cdot  c_*(Y)  = c(TM)\cap c_*(X\cap Y) \in H_*(M)\:.
  \end{equation}
  Similarly
  \begin{equation}
    c_*(M\backslash X) \cdot c_*(M\backslash Y)  = c(TM)\cap c_*(M \backslash (X\cup Y)) \in H_*(M)\:.
  \end{equation}
\end{cor}

The proof of Theorem \ref{main-th} uses first the multiplicativity 
\begin{equation}\label{mult}
 c_*(\alpha \times \beta)= c_*(\alpha)\times c_*(\beta)
\end{equation}
of the MacPherson Chern class transformation (see \cite{K, KYo}), which by induction on the dimension of the support of the constructible functions
and resolution of singularties follows from the Chern class formula
$$c(TM\times TM')= c(TM)\times c(TM')$$
for the Chern classes of  complex (algebraic) manifolds $M,M'$.\\

The second main ingredient (explained in the next section) is a \emph{Verdier-Riemann-Roch Theorem} for the behaviour of MacPherson Chern classes under a 
\emph{non-charac\-teristic} pullback for a morphism $f: M\to N$ of complex (algebraic) manifolds, based on the Lagrangian approach
to MacPherson Chern classes of constructible functions via the \emph{characteristic cycle} map
$$CC: F(Y) \stackrel{\sim} {\to} L(Y,N) $$
to the group $L(Y,N)=L_{an}(Y,N)$  (resp. $L(Y,N)=L_{alg}(Y,N)$) of conic Lagrangian cycles in $T^*N|Y$ for $Y$ a closed subspace 
of the complex (algebraic) manifold $N$. The characteristic cycle map $CC$ is characterized by (see \cite{Sch-book}[(6.35), p.293 and p.323-324])
\begin{equation}\label{CC-Eu}
 CC(Eu_Z)= (-1)^{dim(Z)}\cdot \left[T^*_ZN\right]
\end{equation}
for $Z\subset Y$ a closed irreducible subspace. Here $Eu_Z\in F(Z)$ is the famous \emph{local Euler obstruction} of $Z$, 
with $Eu_Z|Z_{reg}$ constant of value $1$, and $T_Z^*N:=\overline{T_{Z_{reg}}N}$ the closure of the conormal space to the regular part of $Z$.
In particular $CC$ is compatible with switching from the complex algebraic to the complex analytic context.

\begin{thm}\label{VRR}
Let $f: M\to N$ be a morphism of complex (algebraic) manifolds, with $Y\subset N$ a closed subspace and
$X:=f^{-1}(Y)\subset M$. Assume that $\gamma\in F(Y)$ is a constructible function such that $f$ is
\emph{non-characteristic} with respect to the support $supp(CC(\gamma))\subset T^*N|Y\subset T^*N$ of the characteristic cycle
$CC(\gamma)$ of $\gamma$. Then
\begin{equation}
 f^!(c(TN)^{-1} \cap c_*(\gamma)) = c(TM)^{-1} \cap c_*(f^*(\gamma)) \in H_*(X)\:,
\end{equation}
with $f^!: H_*(Y)\to H_*(X)$ the Gysin map induced by the morphism $f: M\to N$ of complex (algebraic) manifolds.
\end{thm}

Also note that
$$ f^!\left( c(TN)^{-1} \cap (-) \right) = f^*(c(TN)^{-1}) \cap f^!(-) = c(f^*TN)^{-1}\cap f^!(-)\:.$$
Before we recall the definition of \emph{non-characteristic}, we need to introduce the following commutative diagram
(whose right square is cartesian, see for example \cite{KS}[(4.3.2), p.199] or \cite{Sch-book}[(4.15), p.249]):
\begin{equation} \label{eq:diag0} \begin{CD}
T^{*}M|X @<t=t_f << f^{*}(T^{*}N|Y) @> f'>> T^{*}N|Y \\
@VV \pi_{X} V  @VV \pi V  @VV \pi_{Y} V \\
X @= X @> f>> Y \:.
\end{CD} \end{equation}
Here $f'$ is the map induced by base change, whereas $t$ is the dual of the differential of $f$.
Then $f$ is by definition \emph{non-characteristic} with respect to a closed conic subset $\Lambda\subset T^*N|Y$
(i.e. a closed complex analytic (or algebraic) subset invariant under the $\bC^*$-action given by multiplication on the fibers of the 
vector bundle $T^*N|Y$), if
\begin{equation}
 f'^{-1}(\Lambda) \cap Ker(t) \subset f^{*}(T^{*}_NN|Y)\:,
\end{equation}
with $f^{*}(T^{*}_NN|Y)$ the zero section of the vector bundle $f^{*}(T^{*}N|Y)$ (compare  also with \cite{KS}[Def.5.4.12] or \cite{Sch-book}[p.255]).\\

If for example $f: M\to N$ is a \emph{submersion}, then $t: f^{*}(T^{*}N|Y) \hookrightarrow T^{*}M|X$ is an injection so that
$Ker(t) = f^{*}(T^{*}_NN|Y)$ is just the zero section of the vector bundle $f^{*}(T^{*}N|Y)$ . So in this case $f$ is non-characteristic
with respect to any  closed conic subset $\Lambda\subset T^*N|Y$.

\begin{cor}
 Let $f: M\to N$ be a submersion of complex (algebraic) manifolds, with $T_f$ the bundle of tangents to the fibers of $f$.
Let $Y\subset N$ a closed subspace and
$X:=f^{-1}(Y)\subset M$. Then
\begin{equation}
c(T_f)\cap f^!(c_*(\gamma)) =  c_*(f^*(\gamma)) \in H_*(X)
\end{equation}
for any $\gamma\in F(Y)$,
with $f^!: H_*(Y)\to H_*(X)$ the Gysin map induced by the smooth morphism $f: X\to Y$.
\end{cor}

Note that $c(TM) = c(T_f) \cup c(f^*TN)$ due to the short exact sequence of vector bundles
$$0\to T_f \to TM \to f^*TN \to 0 \:.$$ 
This Verdier-Riemann-Roch theorem is true for any smooth morphism $f: X\to Y$ of complex (algebraic) varieties
(see  \cite{Yo4}, \cite{FM}[p.111] and \cite{BSY1}[Cor.3.1(3)]). The Verdier-Riemann-Roch theorem for a \emph{smooth} morphism also follows from 
the existence of a corresponding \emph{bivariant} Chern class transformation as in \cite{Br, BSY2}
(as explained in \cite{FM}[Prop.6B, p.67] for a corresponding bivariant Stiefel-Whitney class transformation).
But here in this paper we are interested in the case of a closed  \emph{embedding} $i:M\to N$ of
complex (algebraic) manifolds, which is not covered by the bivariant context (compare \cite{FM}[Sec.10.7, p.112]).

\begin{cor}\label{cor:emb}
 Let $i: M\to N$ be a closed embedding of complex (algebraic) manifolds, with normal bundle $T_MN$.
Let $Y\subset N$ be a closed subspace, with
$X:=i^{-1}(Y)=M\cap Y\subset M$. Assume that $\gamma\in F(Y)$ is a constructible function such that $i$ is
\emph{non-characteristic} with respect to the support $supp(CC(\gamma))$ of the characteristic cycle
$CC(\gamma)$ of $\gamma$. Then
\begin{equation}
i^!(c_*(\gamma)) = c(T_MN) \cap c_*(i^*(\gamma)) \in H_*(X)\:,
\end{equation}
with $i^!: H_*(Y)\to H_*(X)$ the Gysin map induced by the regular embedding $i: M\to N$. 
\end{cor}

Note that $c(i^*TN) = c(TM) \cup c(T^*_MN)$ due to the short exact sequence of vector bundles
$$0\to TM \to i^*TN \to T_MN \to 0 \:.$$ 
Consider for example a stratification of $Y$ by locally closed complex (algebraic) submanifolds $S\subset Y$, which is
\emph{Whitney a-regular}, i.e. such that $$\Lambda:=\bigcup_S T^*_SN\subset T^*N|Y$$ is \emph{closed}.
Then the embedding $i$ is non-characteristic with respect to $\Lambda$, if and only if $M$ is \emph{transversal} to all strata
$S$ of $Y$ (see \cite{Sch-book}[p.255]). So the property \emph{non-characteristic} is a micro-local version (or generalization) of a stratified 
transversality condition. If this is the case, then $i$ is by (\ref{CC-Eu})  
non-characteristic with respect to all $Eu_{\bar{S}}$ and for $Y$ irreducible
(or pure dimensional) also to $Eu_Y$.

If the stratification is also Whitney b-regular (which also implies a-regularity), then 
$$supp(CC(\gamma))\subset \Lambda=\bigcup_S T^*_SN$$
for any $\gamma$ which is constructible with respect to this stratification, i.e. such that
$\gamma|S$ is locally constant for all $S$ (see \cite{Sch-book}[sec.5.0.3).
So if $M$ is transversal to all strata $S$ of a Whitney b-regular stratification, then
$i$ is non-characteristic for all $\gamma$ which are constructible with respect to this stratification.
For example the ambient manifold $N=P^n(\bC)$ is a complex projective space and $M=H$ is a \emph{generic}
hyperplane.

\begin{rem}
 For another approach to Corollary \ref{cor:emb} based on \emph{Verdier specialization} instead of the theory of
 ``non-characteristic pullback for Lagrangian cycles'' see \cite{Sch4}[Cor.0.1, p.7]. 
\end{rem}

In this paper we are especially interested in the \emph{diagonal embedding} $d: M\to M\times M$
for $M$ a complex (algebraic) manifold, so that
$t: T^*M\times_M T^*M \to T^*M$ is just the addition map in the fibers. Also note that the normal bundle
$T_M(M\times M)$ of the diagonal embedding $d$ is isomorphic to $TM$.

\begin{cor}\label{cor:diag}
 Let $M$ be a complex (algebraic) manifold, with $d: M\to M\times M$ the diagonal embedding.
Let $Y\subset M\times M$ be a closed subspace, with
$X:=d^{-1}(Y)\subset M$. Assume that $\gamma\in F(Y)$ is a constructible function such that $d$ is
\emph{non-characteristic} with respect to the support $supp(CC(\gamma))$ of the characteristic cycle
$CC(\gamma)$ of $\gamma$. Then
\begin{equation}\label{VRR-nc}
d^!(c_*(\gamma)) = c(TM) \cap c_*(d^*(\gamma)) \in H_*(X)\:,
\end{equation}
with $d^!: H_*(Y)\to H_*(X)$ the Gysin map induced by the diagonal embedding $d: M\to M\times M$. 
\end{cor}

Consider for example two closed subspaces $X,X'\subset M$ with $Y:=X\times X'$ and $d^{-1}(X\times X')=X\cap X'$.
For $\alpha\in F(X)$ and $\beta\in F(X')$ we have the constructible function $\alpha\times \beta \in F(X\times X')$,
with 
$$\alpha\times \beta((p,p')):=\alpha(p)\cdot \beta(p') \quad \text{for all $p\in X, p'\in X'$.}$$
Then $d^*(\alpha\times \beta)=\alpha \cdot \beta$. Let $\alpha$ resp. $\beta$ be constructible with respect to a 
Whitney b-regular stratification of $X$ resp. $X'$ with strata $S$ resp. $S'$. Then $\alpha\times \beta$ is constructible
with respect to the Whitney b-regular product stratification of $X\times X'$ with strata $S\times S'$.
Assume now that these stratifications of $X$ and $X'$ are \emph{transversal} in the sense that all strata $S$ of $X$
are transversal to all strata $S'$ of $X'$. Then the diagonal embedding $d:M \to M\times M$ is transversal to the 
product stratification of $X\times X'$ so that $d$ is \emph{non-characteristic} with respect to $\alpha\times \beta$.
So by Corollary \ref{cor:diag} and the multiplicativity (\ref{mult}) of the MacPherson Chern class one  gets in this 
way a proof of Theorem \ref{main-intro}. with
\begin{equation} \begin{split}
 d^!(c_*(\alpha) \times c_*(\beta)) &= d^!(c_*(\alpha\times \beta)) \\
&= c(TM) \cap c_*(\alpha\cdot \beta)) \in H_*(X\cap X')\:.
\end{split}\end{equation}

Theorem \ref{main-th} states the same result under the assumption that $\alpha$ and $\beta$ are 
\emph{splayed}, which should be seen as another transversality assumption, similar but different from the 
\emph{stratified transversality} used above. In fact, splayedness implies \emph{locally} in the analytic topology
this stratified transversality: Assume $M=V_1\times V_2$ is a product
 of complex analytic manifolds, with $\alpha=\pi_1^*(\alpha')$ and $\beta=\pi_2^*(\beta')$ for some $\alpha'\in F(V_1)$ and 
 $\beta'\in F(V_2)$, with $\pi_i: V_1\times V_2\to V_i$ the projection ($i=1,2$).
Take a Whitney b-regular stratification of $V_1$ resp. $V_2$ with strata $S$ resp. $S'$, so that $\alpha'$ resp. $\beta'$
are constructible with respect to them. Then $\alpha$ resp. $\beta$ is constructible with respect to the two
Whitney b-regular stratifications of $M$ with strata $S\times V_2$ resp. $V_1\times S'$,
which (trivially) intersect transversaly. And this implies again the global non-characteristic property:
\begin{center}
 splayed $\quad \Rightarrow \quad $ local stratified transversality \quad $ \Rightarrow \quad $ non-characteristic.
\end{center}

Then the proof of Theorem \ref{main-th}
 follows (as before) from Corollary \ref{cor:diag} and the multiplicativity
(\ref{mult}) of the MacPherson Chern class.
In some sense this implication together with Corollary \ref{cor:diag} give a ``mechanism
obtaining intersection-theoretic identities from local analytic data''
as asked for in \cite{AF}[Rem.3.5].\\

In the following we give an even simpler proof of the implication 
 \begin{center}
 splayed $\quad \Rightarrow \quad $  non-characteristic,
\end{center}
which also applies to other interesting situations. To shorten the notation, 
and also to better emphasize the underlying principle of proof, let us introduce 
the \emph{closed conic} subset
$$co(\alpha):=supp(CC(\alpha))\subset T^*M|X\subset T^*M$$
for $\alpha\in F(X)$ with $X$ a closed subspace of the complex (algebraic) manifold.
The three formal properties we need are:
\begin{itemize}
 \item[(co1)] $co$ is locally defined in the sense that it commutes with restriction to open submanifolds of $M$.
 \item[(co2)] We have the multiplicativity 
 $$co(\alpha\times \beta) \subset co(\alpha)\times co(\beta) \subset T^*M\times T^*M=T^*(M\times M).$$
 \item[(co3)] For the projection $\pi: M\times N\to M$ from the product of two manifolds one has
 $$co(\pi^*(\alpha))\subset co(\alpha) \times T^*_NN \subset T^*M\times T^*N=T^*(M\times N),$$
 with $ T^*_NN$ the zero section of $T^*N$.
\end{itemize}
Note that in the case of  $co(\alpha):=supp(CC(\alpha))$ the support of 
the characteristic cycle $CC(\alpha)$ of a constructible function, these properties follow for
example from (\ref{CC-Eu}) together with the multiplicativity (which holds in the algebraic context over a base field of
characteristic zero):
\begin{equation} \label{mult-eu}
 Eu_{Z\times Z'}= Eu_Z \times Eu_{Z'}\:.
\end{equation}

\begin{lem}\label{splayed-non-ch}
 Let $X, Y$ be two closed subspaces of the complex (algebraic) manifold $M$ with given 
 \emph{splayed} constructible functions $\alpha\in F(X)$ and  $\beta \in F(Y)$.
  Then the diagonal embedding $d: M\to M\times M$ is \emph{non-characteristic} with respect
  to $co(\alpha\times \beta)$.
\end{lem}

\begin{proof}
 Note that even if we are working in a complex algebraic context, the non-characteristic property in this
 context follows already from the corresponding non-characteristic property of the associated complex analytically constructible functions.
 Moreover, by (co1) this property can be locally checked in the analytic topology. By the splayedness condition, we can assume
 $M=V_1\times V_2$ is a product
 of analytic manifolds, with $\alpha=\pi_1^*(\alpha')$ and $\beta=\pi_2^*(\beta')$ for some $\alpha'\in F(V_1)$ and 
 $\beta'\in F(V_2)$, with $\pi_i: V_1\times V_2\to V_i$ the projection ($i=1,2$).
 But then one gets by (co3):
 $$co(\alpha)\subset co(\alpha')\times T^*_{V_2}V_2 \subset T^*V_1\times T^*_{V_2}V_2$$
 and
 $$co(\beta)\subset  T^*_{V_1}V_1 \times co(\beta')) \subset T^*_{V_1}V_1 \times T^*V_2\:.$$
Assume now that $t((p,\omega),(p,\omega'))= (p,\omega+\omega')=0\in T^*M|p$ for some
$$((p,\omega),(p,\omega'))\in co(\alpha\times \beta) \subset co(\alpha)\times co(\beta) \:.$$ 
Then $\omega'=-\omega$ and 
$$(p,\omega) \in co(\alpha)\cap a_*(co(\beta)) \:,$$
with $a: T^*M\to T^*M$ the antipodal map. 
Therefore $(p,\omega)$ belongs to
 $$ \left( T^*V_1\times T^*_{V_2}V_2 \right)  \cap \left( T^*_{V_1}V_1 \times T^*V_2 \right) = T^*_{V_1}V_1 \times T^*_{V_2}V_2 \:,$$
 the zero section of $T^*M$. 
\end{proof}

Let us finish this section with some other interesting examples, where a similar notion of splayedness can be introduced.

\begin{ex}
 Let us consider the MacPherson Chern class transformation $c_*: F(X)\to CH_*(X)$ in the embedded algebraic context over a base 
 field of characteristic zero (see e.g. \cite{Ken1, Sch-lect}). Then one can introduce the notion of splayedness for two constructible functions as before,
 but asking the corresponding ``splitting'' locally in the Zariski- or \'{e}tale topology (which in the complex algebraic context is of course much 
 stronger than working locally in the analytic topolgy). Using the corresponding characteristic cycle map $CC$ characterized by 
 (\ref{CC-Eu}), all of the results and their proofs of this and the following chapter apply. For the important comparison
 of the non-characteristic
 pullback of characteristic cycles with the corresponding pullback for constructible function
as in Theorem \ref{non-cc}, one can use for example the Lefschetz principle to reduce
 this to the complex algebraic context studied here.
\end{ex}

\begin{ex}
Consider the embedded real  subanalytic  or semi-algebraic  context, with $X$ a closed subanalytic (or semi-algebraic) subset in an
ambient real analytic (Nash-) manifold $M$, and $F(X,\bZ_2)$ the corresponding group of
$\bZ_2$-valued constructible functions. Then one can consider (as in \cite{FuMC,  Sch-lect}) the corresponding characteristic cycle
map 
$$CC_2: F(X,\bZ_2)\stackrel{\sim}{\to} L(X,M)\otimes \bZ_2$$
to conic subanalytic (or semi-algebraic) Lagrangian cycles with $\bZ_2$ coefficients 
in $T^*M|X$. Since we are working with $\bZ_2$-coefficients, no orientability of $M$ is needed.
But here conic just means $\bR^+$-invariant. For the Lagrangian approch to Stiefel-Whitney classes of such constructible functions, it is important to work
only with those constructible functions, for which  $CC_2(\alpha)$ is also $\bR^*$-invariant so that it can be projectivised as in the complex
context. This just means $a_*CC_2(\alpha)=CC_2(\alpha)$ for the antipodal map $a: T^*M\to  T^*M$, and corresponds by a beautiful observation
of Fu and McCrory (\cite{FuMC}) to the classical \emph{local Euler condition} of Sullivan (a sort of self duality).
Introducing the local splayedness condition as before, one gets first the important property that for
$\alpha$ and $\beta$ \emph{splayed} and both satisfying the \emph{local Euler condition}, also $\alpha\cdot \beta$ satisfies the 
local Euler condition. For example if $X,Y$ are two splayed closed subanalytic (or semi-algebraic) \emph{Euler subspaces} of $M$
(i.e. such that $1_X,1_Y$ satisfy the local Euler condition), then also the intersection $X\cap Y$ is an Euler space.

 Let $F^{Eu}(X;\bZ_2)\subset F(X,\bZ_2)$ denote the corresponding subgroup of constructible functions satisfying
the ``local Euler condition''. Then one can introduce the Stiefel-Whitney class transformation
$$w_*: F^{Eu}(X;\bZ_2)\to H^{BM}_*(X,\bZ_2)$$
with the help of the corresponding $\bR^*$-invariant Lagrangian cycles very similar to the approach to the MacPherson Chern class $c_*$  in the complex context
discussed in the next chapter (see \cite{FuMC,  Sch-lect}). Then all our results and proofs also apply to this context 
(as will be explained somewhere else). 

If for example  $X,Y$ are two splayed closed subanalytic (or semi-algebraic) \emph{Euler subspaces} of $M$, then
\begin{equation}
w_*(X) \cdot  w_*(Y)  = w(TM)\cap w_*(X\cap Y) \in H^{BM}_*(M,\bZ_2)\:.
\end{equation}
For a similar result about the Stiefel-Whitney class of the transversal
intersection of two Euler spaces in the $pl$-context compare with
\cite{Matsui}.
\end{ex}

\begin{ex}
 We consider the micro-local view on sheaf theory developed by Kashiwara and Schapira \cite{KS}.
 Here $M$ is a real analytic manifold, and $F,G\in D^b(M)$ are bounded complexes of sheaves.
 Let us call them \emph{splayed}, if there is near any given point $p\in M$ a local analytic isomorphism $M=V_1\times V_2$
 of analytic manifolds so that $F=\pi_1^*(F')$ and $G=\pi_2^*(G')$ for some $F'\in D^b(V_1)$ and 
 $G'\in D^b(V_2)$, with $\pi_i: V_1\times V_2\to V_i$ the projection ($i=1,2$). Here the pullback just means the usual sheaf theoretic
 pullback. 
 
 Let for example $F\in D^b(M)$ be a sheaf complex, all of whose cohomology sheaves are locally constant.
Then $F$ and $G$ are splayed for all $G\in D^b(M)$.
 
 Back to the general case, assume only that $F$ and $G$ are splayed. 
 Then the diagonal embedding $d: M\to M\times M$ is \emph{non-characteristic} with respect to the \emph{micro-support}
 $$ SS(F\boxtimes^L G)\subset T^*(M\times M)$$ of $F\boxtimes^L G$. In fact our proof of Lemma \ref{splayed-non-ch}
 applies to the closed $\bR^+$-conic subset $co(F):=SS(F)\subset T^*M$, which satisfies the three properties (co1-3)
 by \cite{KS}[Prop.5.4.1, Prop.5.4.5].

Assume that we are in the complex analytic context, with $F,G\in D^b_c(X)$ constructible sheaf complexes
(with rational  coefficients),
which are perverse (up to a shift) for the \emph{middle perversity} (see \cite{KS}[Sec.10.3] or \cite{Sch-book}[Chapter 6]).
If $F$ and $G$ are splayed, then also 
$$F\otimes^L G = d^*( F\boxtimes^L G )$$
is  a perverse sheaf up to a shift
(compare also with \cite{KS}[Lem.10.3.9]), since 
\begin{equation}
 F\otimes^L G =
d^*( F\boxtimes^L G ) \simeq d^!( F\boxtimes^L G )[2\cdot dim_C(M)]
\end{equation}
due to the non-characteristic property (see \cite{KS}[Cor.5.4.11]).
Let for example $X,Y$ be two \emph{splayed} pure dimensional complex analytic subsets.
Then also the corresponding middle perversity \emph{intersection cohomology complexes} $IC_X,IC_Y$ are splay\-ed, 
since these  commute with smooth pullback
% (up to a shift, 
(see  \cite{Sch-book}[Sec.6.0.2, Lem.6.0.3].
Here we are using the convention that $IC_X=\bQ_X$ for $X$ smooth).
 Then also $X\cap Y$ is pure dimensional of dimension
$dim_C(X)+dim_C(Y)-dim_C(M)$, with
$$IC_{X\cap Y}\simeq  d^*(  IC_X \boxtimes^L IC_Y )\:.$$
 If $X,Y$ are also compact, then one can also consider the corresponding Goresky-MacPher\-son L-classes
\cite{GM}, and it becomes natural to ask for the following intersection formula:
\begin{equation}\label{int-L}
L_*(X) \cdot  L_*(Y)  \stackrel{?}{=} L(TM)\cap L_*(X\cap Y) \in H_*(M,\bQ)\:,
\end{equation}
with $L(TM)$ the Hirzebruch L-class.
The Lagrangian approach taken up in this paper doesn't apply to the theory of 
$L$-classes of singular spaces. But the method of proof developed in
\cite{Sch4} based on \emph{Verdier specialization} might work,
since the non-characteristic property of the diagonal embedding  $d$ with respect to
the micro-support
$SS(F\boxtimes^L G)$ implies by \cite{KS}[Cor.5.4.10(i)]
that the support 
$$supp(\mu_M(F\boxtimes^L G))$$
 of the \emph{microlocalization} $\mu_M(F\boxtimes^L G)$ 
of two splayed sheaf complexes is contained in
the zero-section of $T^*_M(M\times M)$.

For a result  similar to (\ref{int-L}) about the Goresky-MacPherson L-class of a transversal
intersection in the $pl$-context  compare with
\cite{Matsui2}.
\end{ex}

\begin{ex}
We are working with coherent left $\cD$-modules on the complex analytic manifold $M$
(see e.g. \cite{KS}[Chapter XI]).
 Let us call two such modules $\cF,\cG$ \emph{splayed}, if there is near any given point $p\in M$ a local analytic isomorphism $M=V_1\times V_2$
 of analytic manifolds so that $\cF=\pi_1^{-1}(\cF')$ and $\cG=\pi_2^{-1}(\cG')$ for some $\cD$-modules $\cF'$ on $V_1$  and 
 $\cG'$ on $V_2$, with $\pi_i: V_1\times V_2\to V_i$ the projection ($i=1,2$). Here the pullback means this time the 
 pullback of left $\cD$-modules. 
 Then the diagonal embedding $d: M\to M\times M$ is \emph{non-characteristic} with respect to the \emph{characteristic variety}
 $$ char(\cF\underline{\boxtimes}  \cG)\subset T^*(M\times M)$$ 
of the $\cD$-module cross product $\cF\underline{\boxtimes}  \cG$. In fact our proof of Lemma \ref{splayed-non-ch}
 applies to the closed complex analytic 
 $\bC^*$-conic subset 
 $$co(\cG):=char(\cG)\subset T^*M \:,$$ 
 which satisfies the  properties (co1-3)
 by \cite{KS}[(11.2.22), Prop.11.2.12]. Then also
\begin{equation*}
\cF \underline{\otimes} \cG :=
d^{-1}( \cF \underline{\boxtimes}  \cG ) 
\end{equation*}
is a \emph{coherent} $\cD$-module by the non-characteristic property (see  \cite{KS}[Prop.11.2.12]).
If we are working with \emph{holonomic} $\cD$-modules, then their characterictic varieties are 
also closed conic \emph{Lagrangian} subsets of $T^*M$. Moreover, one can also directly define
for a holonomic $\cD$-module $\cG$ its \emph{characteristic cycle} $CC(\cG)$ supported on $char(\cG)$.
And it was  Ginsburg \cite{Gi2}, who first used from this $\cD$-module view point the ``non-characteristic''
condition in his study of the behavior of the MacPherson Chern class transformation with respect to a suitable 
convolution product (and compare also with \cite{Gi1} for his $\cD$-module approach to the
 MacPherson Chern class transformation).

Finally one can similarly introduce the notion of splayedness for $F\in D^b(M)$ a bounded sheaf complex
(with complex coefficients) and $\cG$ a coherent left $\cD$-module, using the usual (resp. $\cD$-module) pullback for
$F$ resp. $\cG$. Then the closed conic subsets $co(F)=SS(F) \subset T^*M$ and 
$co(\cG)=char(\cG)\subset T^*M$ both satisfy the properties (co1) and (co3).
Therefore one gets for $F$ and $\cG$ \emph{splayed} as in the proof 
 of Lemma \ref{splayed-non-ch} the estimate
\begin{equation}
SS(F)\cap char(\cG) \subset T^*_MM.
\end{equation}
So if $F$ is also a subanalytically constructible sheaf complex, one gets that  
$(F,\cG)$ is an \emph{elliptic pair} in the sense of \cite{ScSc}.
\end{ex}

%%%%%%%%%%%%%%%%%%%%%%%%%%%%%%%%%%%%%%%%%%%%%%%%%%%%%%%%%
%%%%%%%%%%%%%%%%%%%%%%%%%%%%%%%%%%%%%%%%%%%%%%%%%%%%%%%%%

\section{Pullback of Lagrangian cycles}
In this section we work again in the embedded complex analytic or algebraic
context, with $Y$ a closed subspace in the complex manifold $N$.
Let us first recall the main ingredients in MacPherson's definition of his
(dual) Chern classes of a constructible function and the well known by now
relation to the \emph{theory of characteristic cycles}
(cf. \cite{Gi1, Ken1, Sab, Sch-lect}).
 Then the main characters
of this story can be best visualized in the commutative diagram

\begin{equation} \label{eq:conormal}
\begin{CD}
F(Y) @< \check{E}u < \sim < Z_*(Y) @> \check{c}^{Ma}_{*} >> H_{*}(Y) \\
@|  @V cn V \wr V @| \\
F(Y) @ > CC > \sim > L(Y,M) @> c(T^{*}M|Y)\cap s_{*} >> H_{*}(Y) .
\end{CD}
\end{equation}
\noindent
Here $F(Y)$ and $Z_*(Y)$ are the groups of constructible functions and cycles
in the corresponding complex analytic or algebraic context
(where we allow locally finite sums in the analytic context). Similarly
 $H_{*}(Y)$ denotes  the Borel-Moore homology group in even degrees $H^{BM}_{2*}(Y,\bZ)$
or in the algebraic context the Chow group $CH_{*}(Y)$.\\

The transformation $\check{E}u$ associates to a closed irreducible subset $Z$ of $Y$
the constructible function given by the \emph{dual Euler obstruction}
$$\check{E}u_{Z}:=(-1)^{dim(Z)}\cdot Eu_{Z}$$
of $Z$,
and is linearly extended to cycles. Then $\check{E}u$ is an isomorphism of groups,
since $Eu_{Z}|Z_{reg}$ is constant of value $1$. The transformation
$\check{c}^{Ma}_{*}$ is similarly defined by associating to an irreducible $Z$
the total dual Chern-Mather class $\check{c}^{Ma}_{*}(Z)$ of $Z$ viewed
in $H_{*}(Y)$. 
One has the following
description of the dual Chern-Mather class of $Z$ 
in terms of the Segre class of the 
\emph{conormal space} $T_{Z}^{*}N:=\overline{T_{Z_{reg}}^{*}N}$ of $Z$ in $N$, which is 
 a conic Lagrangian subspace in $T^{*}N|X$:
\begin{equation}\label{cor:dualMather}
\check{c}^{Ma}_{*}(Z)=c(T^{*}N|Z)\cap s_{*}(T_{Z}^{*}N) 
\end{equation}
 after e.g. \cite[Lemme (1.2.1)]{Sab} or \cite[Lemma 1]{Ken1}.
 The Segre class is defined by (compare  \cite[Sec.4.1]{Fu}):
\begin{equation}\label{eq:segre} \begin{split}
s_{*}(T_{Z}^{*}N ) &:=\hat{\pi}_{*}(c(\cO(-1))^{-1}\cap [\bP(T_{Z}^{*}N \oplus \mbox{\boldmath $1$}  )]) \\
&= \sum_{i\geq 0}\: \pi_{*}(c^{1}(\cO(1))^{i}\cap [\bP(T_{Z}^{*}N \oplus \mbox{\boldmath $1$} )]).
\end{split}
\end{equation}

Here $\cO(-1)$ denotes the tautological line subbundle on the projective completion
$\hat{\pi}:  \bP(T^{*}N|_Z  \oplus \mbox{\boldmath $1$})\to Z$
with $\cO(1)$ as its dual. Note that we work with the projective completion to loose no information contained in the zero section.

\begin{ex} \label{ex:smooth}
Let $Z$ be a closed complex submanifold of $N$,
and consider the Lagrangian cycle $[T^{*}_{Z}N]$. Then $s_{*}([T^{*}_{Z}N])=
(c(T^{*}_{Z}N))^{-1}\cap [Z]$ and
\[c(T^{*}N|Z)\cap s_{*}([T^{*}_{Z}N])
= c(T^{*}N|Z)\cap (c(T^{*}_{Z}N))^{-1}\cap [Z]
= c(T^{*}Z)\cap [Z] .\]
Here we use the Whitney formula for the total Chern class $c$ and the exact sequence of vector 
bundles:
\[0\to T^{*}_{Z}N \to T^{*}N|Z \to T^{*}Z \to 0 .\]
\end{ex}
 
  By definition,
$L(Y,N)$ is the group of all cycles generated by the conormal spaces $T_{Z}^{*}N$
for $Z\subset Y$ closed and irreducible.
The vertical map $cn$ in diagram (\ref{eq:conormal}) is  the correspondence $Z\mapsto T_{Z}^{*}N$. 
Then (\ref{cor:dualMather}) obviously implies the commutativity of the right square in (\ref{eq:conormal}).

The \emph{dual MacPherson Chern class transformation} is defined by
\begin{equation}\label{dual-c}
\check{c}_{*}:=\check{c}^{Ma}_{*}\circ \check{E}u^{-1}:\: 
F(Y)\to H_{*}(Y).
\end{equation}
This agrees up to a sign with MacPherson's original definition \cite{MP}
of his Chern class transformation $c_{*}$, namely
\begin{equation}
\check{c}_{i}(\alpha) = (-1)^{i}\cdot c_{i}(\alpha) \in H_{i}(Y) \: .
\end{equation}

The commutativity of the left square in (\ref{eq:conormal})
follows either by definition, if the characteristic cycle map $CC$ is defined 
(as done in \cite{Ken1, Sab}) by
\begin{equation}\label{CC-Eu:2}
 CC(\check{E}u_Z) =  \left[T^*_ZN\right]
\end{equation}
for $Z\subset Y$ a closed irreducible subspace. Working in the complex algebraic or analytic context in the classical topology,
one can also take another more refined definition based on \emph{stratified Morse theory for constructible functions}
(as done in \cite{Sch-lect,ST}  and \cite{Sch-book}[Sec.5.0.3]):
\begin{equation}\label{Morse}
 CC(\alpha):= \sum_S \: (-1)^{dim(S)}\cdot \chi((NMD(S),\alpha))\cdot \left[T^*_{\bar{S}}N\right]
\end{equation}
for $\alpha\in F(Y)$ constructible with respect to a given Whitney b-regular stratification of $Y$ with connected strata $S$.
Here $\chi((NMD(S),\alpha))$ is the Euler characteristic of a suitable \emph{normal Morse datum} $NMD(S)$ of $Y$ weighted by
$\alpha$, which only depends on the stratum $S$ (but the details are not needed in this paper).
Then the equality (\ref{CC-Eu:2}) 
follows from  \cite{Sch-book}[(6.35), p.293 and p.323-324].\\

Let us consider a morphism of manifolds $f: M\to N$ as in the diagram 
(\ref{eq:diag0}), with $X:=f^{-1}(Y)$ and the same notations as in (\ref{eq:diag0}),
e.g. with  the induced map $f': f^{*}(T^{*}N|Y)\to T^{*}N|Y$.
Then we get a similar diagram for the \emph{projective completions}:

\begin{equation}\label{compl} \begin{CD}
\bb{P}(T^{*}M|X\oplus\mbox{\boldmath $1$}) @<\hat{t} << 
U\subset \bb{P}(f^{*}(T^{*}N|Y)\oplus\mbox{\boldmath $1$}) @> \hat{f}>> 
\bb{P}(T^{*}N|Y\oplus\mbox{\boldmath $1$}) \\
@VV \hat{\pi}_{X} V  @VV \hat{\pi} V  @VV \hat{\pi}_{Y} V \\
X @= X @> f>> Y .
\end{CD} \end{equation}
The right square is cartesian, but the map $\hat{t}$ is
only defined on the complement $U$ of $\bb{P}(Ker(t)\oplus\{0\})$.
For the application to Segre-classes it is important to note,
that 
\begin{equation}\label{iso}
\hat{t}^{*}({\mathcal O}_{X}(1))\simeq 
(\hat{f}^{*}({\mathcal O}_{Y}(1)))|U\:.
 \end{equation}
One has the following characterization (compare with \cite{Sch-book}[Lem.4.3.1] for a counterpart in the real algebraic
resp. analytic context):
\begin{lem}\label{nc-proper}
 Let $C\subset T^*N|Y$ be a closed conic complex algebraic resp. analytic subset. Then the following conditions are equivalent:
 \begin{enumerate}
  \item $f: M\to N$ is \emph{non-characteristic} for $C$, i.e. 
   $f'^{-1}(C) \cap Ker(t)$ is contained in the 
zero section $f^{*}(T^{*}_NN|Y)$  of the vector bundle $f^{*}(T^{*}N|Y)$.
\item The map $t: f'^{-1}(C) \to T^*M$ is \emph{proper} and therefore finite
(since its fibers are subspaces of an affine  vector space).
 \end{enumerate}
\end{lem}

\begin{proof}
Note that $(2) \Rightarrow (1)$ is obvious, since $C$ is conic.
So lets discuss the other implication.
 Looking at the projective completions, one gets a commutative diagram
\begin{equation} \begin{CD}
T^{*}M|X @< t <<  f'^{-1}(C) @> f' >>  C \\
@VVV  @VVV @VVV \\
\bb{P}(T^{*}M|X\oplus\mbox{\boldmath $1$}) @<\hat{t} << 
\hat{f}^{-1}(\hat{C})  @> \hat{f}>> 
\hat{C} \\
@VV \hat{\pi}_{X} V  @VV \hat{\pi} V  @VV \hat{\pi}_{Y} V \\
X @= X @> f>> Y .
\end{CD}  \end{equation}
The upper vertical maps are the natural inclusions, and all squares except the
lower left square are cartesian.
Note that $\hat{t}$ is defined on $\hat{f}^{-1}(\hat{C})$,
since $\hat{f}^{-1}(\hat{C})\subset U$ by the assumption (1).
Then $\hat{t}|\hat{f}^{-1}(\hat{C})$ is proper, since $\hat{\pi}_{X}
\circ \hat{t} = \hat{\pi}$ and $\hat{\pi},\hat{\pi}_{X}$ are proper.
But then $t|f'^{-1}(C)$ is proper by base change.
\end{proof}

So if $f: M\to N$ is non-characterstic with respect to the closed conic complex algebraic resp. analytic subset $C\subset T^*N|Y$,
we can define for the closed conic complex algebraic resp. analytic subset $C':=t(f'^{-1}(C))\subset T^*M|X$ the induced group homomorphism
\begin{equation}\label{pullback-cone}
t_*\circ f'^{!}: H_*(C)\to H_*(C')\:.
\end{equation}
Here we use the map $f':  T^*N\to f^*T^*N$ of ambient complex (algebraic) manifolds
for the refined Gysin map $f'^!: H_*(C)\to H_*(f^{-1}(C))$.
Note that $t_*$ is degree preserving, whereas 
$$f'^!: H_i(C)\to H_{i+m-n}(f^{-1}(C))\:,$$
with $n=dim(N), m=dim(M)$ and $m-n=dim(f^*T^*N)-dim(T^*N)$.
Assume that $C$ resp. $C'$ are pure $d$- resp. $d'$-dimensional, with $d':=d+m-n$
(as will be the case for $C$ Lagrangian with $d=n$ and $d'=m$).
Then we get an induced pullback map of cycles:
$$t_*\circ f'^!: Z_d(C)\simeq H_d(C)\to H_{d'}(C')\simeq Z_{d'}(C') $$
behaving nicely with respect to \emph{Segre classes}:
\begin{equation}\label{pullback-s}
f^!(s_*([C]))= s_*(t_*f'^![C]) \in H_*(X)\:.
\end{equation}

In fact, one has the following sequence of equalities
(with the maps of projective completions as in (\ref{compl})):

\begin{equation*}\begin{split}
f^!s_*([C]) &= \sum_{i\geq 0} f^{!}\bigl(\hat{\pi}_{Y*}
(c^{1}({\mathcal O}_{Y}(1))^{i} \cap [\hat{C}])\big) \\
&= 
\sum_{i\geq 0}  \hat{\pi}_*\bigl(\hat{f}^!
(c^{1}({\mathcal O}_{Y}(1))^{i} \cap [\hat{C}])\big) \\
&= 
 \sum_{i\geq 0} \hat{\pi}_*
\bigl(c^1(\hat{f}^*{\mathcal O}_{Y}(1))^i \cap (\hat{f}^![\hat{C}])\big) \\
&=
\sum_{i\geq 0} \hat{\pi}_{X*}\hat{t}_*
\bigl(c^{1}(\hat{t}^*{\mathcal O}_{X}(1))^i \cap (\hat{f}^![\hat{C}])\big) \\
&= \sum_{i\geq 0} \hat{\pi}_{X*}\bigl(c^1({\mathcal O}_{X}(1))^i \cap
(\hat{t}_*\hat{f}^![\hat{C}])\big) =
s_{*}(t_*f^![C]) \:.
\end{split} \end{equation*}

Here we are using:
\begin{enumerate}
\item the base change isomorphism $f^!\hat{\pi}_{Y*} = \hat{\pi}_{*}\hat{f}^!$,
\item the compability  $\hat{f}^!(c^1(\cdot)\cap (\cdot)) = c^1(\hat{f}^{*}(\cdot))
\cap \hat{f}^!(\cdot)$,
\item the isomorphism (\ref{iso}) together with  $\hat{\pi}_*=\hat{\pi}_{X*}\hat{t}_*$ by the
functoriality of pushdown,
\item the projection formula $\hat{t}_{*}(c^1(\hat{t}^{*}\cdot)\cap (\cdot)) =
c^1(\cdot)\cap \hat{t}_{*}(\cdot)$.
\end{enumerate}

The non-characteristic pullback map (\ref{pullback-cone}) is also functorial in $f$.
Let $g: V\to M$ be another morphism of complex (algebraic) manifolds,
with $Z:=g^{-1}(X)=(f\circ g)^{-1}(Y)$. Consider the cartesian diagram

\begin{equation}\begin{CD}
g^*f^*T^*N @> g'' >> f^*T^*N @> f' >> T^*N \\
@V t'_f VV @VV t_f V \\
g^*T^*M @>> g' > T^*M \\
@V t_g VV \\
T^*V \:,
\end{CD}\end{equation}

with
$$f'\circ g''= (f\circ g)': g^*f^*T^*N = (f\circ g)^*T^*N \to T^*N$$
and 
$$t_g\circ t'_f = t_{f\circ g}: g^*f^*T^*N = (f\circ g)^*T^*N \to T^*V\:.$$

Using Lemma \ref{nc-proper}, one easily gets that $f\circ g$ is non-characteristic with respect to the
 closed conic complex algebraic resp. analytic subset $C\subset T^*N|Y$,
if and only if $f$ is non-characteristic with respect to $C$ and $g$ is non-characteristic with respect to
$C':=t(f'^{-1}(C))\subset T^*M|X$. Moreover, in this case one gets

\begin{equation}\label{funct-pb}
(t_{g*}g'^!) \circ (t_{f*}f'^!)= t_{g*}t'_{f*}g''^!f'^!= t_{f\circ g*}(f\circ g)^!: H_*(C)\to H_*(C'')\:,
\end{equation}
with $C'':= t_g(g'^{-1}(C'))=t_{f\circ g}(f\circ g)'^{-1}(C) \subset T^*V|Z$. Here we are using the 
functorialities $t_{g*}t'_{f*}=t_{f\circ g*}$ and $g''^!f'^!=(f\circ g)'^!$ together with the base change isomorphism
$g'^!t_{f*}=t'_{f*}g''^!$.\\

Now we can come to the main result of this chapter.
\begin{thm}\label{non-cc}
Let  $f: M\to N$ be a morphism of complex (algebraic) manifolds of (complex) dimensions $m=dim(M), n=dim(N)$,
and $Y\subset N$ be a closed subspace, with $X:=f^{-1}(Y)\subset M$.
Assume that $f$ is non-characteristic with respect to the support
$C:=supp(CC(\gamma))\subset T^*N|Y$ of the characteristic cycle $CC(\gamma)$ of a constructible
function $\gamma\in F(Y)$. Then $C':=t_{f}(f'^{-1}(C))$ is pure $m$-dimensional,
with
\begin{equation}
t_{f*}f'^!(CC(\gamma)) = (-1)^{m-n}\cdot CC(f^*(\gamma))\:.
\end{equation}
In particular,  the left hand side is a Lagrangian cycle in $T^*M|X$.
\end{thm}

So by (\ref{cor:dualMather}), (\ref{dual-c}) and (\ref{pullback-s})  we get 
\begin{equation}\begin{split}
f^!(c(T^*N)^{-1}\cap \check{c}_*(\gamma))  &= f^!s_*(CC(\gamma)) \\
&= (-1)^{m-n}\cdot s_*(CC(f^*(\gamma))) \\
 &= (-1)^{m-n}\cdot c(T^*M)^{-1}\cap \check{c}_*(f^*(\gamma)) \:.
\end{split}\end{equation}

And this is just a reformulation of our main  Theorem \ref{VRR} in terms of the
dual Chern MacPherson class, since $c^i(V^*)=(-1)^i\cdot c^i(V)$ for a complex vector bundle $V$.
Recall that also $\check{c}_i(\gamma)=(-1)^i\cdot c_i(\gamma)$. Finally the sign $(-1)^{m-n}$ is disapearing in
 Theorem \ref{VRR}, since $t_{f*}f'^!: H_i(C)\to H_{i+m-n}(C')$.

\begin{rem}
Theorem \ref{non-cc} is similar to the result of \cite{KS}[Prop.9.4.3] for constructible sheaf complexes in the context of real geometry,
but formulated under the stronger assumption of the non-characteristic property with respect to the \emph{micro-support}
instead of the  support of the characteristic cycles. Moreover, if one then applies \cite{KS}[Prop.9.4.3] to our complex analytic
(or algebraic) context, then also the sign $(-1)^{m-n}$ is not appearing, because a different orientation convention is used
for the ambient manifolds $T^*N,T^*M$ (needed for the definition of the Gysin map $f'^!$.
 Compare \cite{Sch-book}[Rem.5.0.3]).
\end{rem}

Let us now discuss the proof of Theorem \ref{non-cc}. By using the graph embedding and the functoriality
(\ref{funct-pb})  of the non-characteristic pullback, one can consider separately the case of a submersion $f: M\to N$ of complex
(algebraic) manifolds, and that of a closed embedding $i: M\to N$. Moreover it can be (\'{e}tale)  locally checked on $X\subset M$.
So for the case of a submersion we can assume $f: M\times N\to N$ is a projection of complex manifolds. Then the claim follows from
$$t_{f*}f'^{-1}([T^*_ZN])= [T^*_MM] \times [T^*_ZN]$$
and the mutiplicativity of the Euler obstruction
$$Eu_{M\times Z}= Eu_M\times Eu_Z = 1_M \times Eu_Z = f^*(Eu_Z)$$
for $Z\subset Y$ a closed irreducible subspace. Then 
$$CC(Eu_Z)=(-1)^{dim(Z)}\cdot [T^*_ZN] $$
and
$$CC(Eu_{M\times Z})=(-1)^{dim(Z)+dim(M)}\cdot  [T^*_MM] \times [T^*_ZN]\:.$$
This also explains the sign $(-1)^{dim(M)}$ appearing, with $dim(M)$ the fiber dimension of the projection $f: M\times N\to N$.
Here the multiplicativity of the Euler obstruction is also equivalent to the multiplicativity of the characteristic cycle map
$CC$. If one uses the refined Morse theoretical definition (\ref{Morse}), then this multiplicativity follows from 
\cite{Sch-book}[(5.6) and (5.24)], and the sign $(-1)^{dim(M)}$ comes from the use of $(-1)^{dim(S)}$ in the definition
(\ref{Morse}).\\

The local study  of a closed embedding $i: M\to N$ can be reduced by induction (and the functoriality
(\ref{funct-pb})  of the non-characteristic pullback) to the case $i: M=\{f=0\}\hookrightarrow N$ of the inclusion of a smooth 
hypersurface, with $f: N\to \bC$ a submersion. Assume $i$ is non-characteristic for $T^*_ZN$, with $Z\subset Y \subset  N$ a closed
irreducible subspace. This just means $df(M)\cap T^*_ZN=\emptyset$. By shrinking $N$, we can also assume
$df(N)\cap T^*_ZN=\emptyset$.
Then $f: Z_{reg}\to \bC$ is also a submersion,
in particular $Z\not\subset M=\{f=0\}$. Consider the exact sequence of vector bundles on $N$:
$$0\to <df>= Ker(p)\to T^*N \to T_f^*\to 0\:,$$
with the projection $p: T^*N\to T^*_f$ dual to the inclusion $T_f\to TN$ of the subvector bundle of tangents to the fibers of $f$.
Then $T^*_ZN\cap Ker(p)$ is contained in the zero-section $T_N^*N$ of $T^*N$.
Therefore $p: T^*_Z\to T^*_f$ is proper and finite (as in the proof of Lemma \ref{nc-proper}),
with 
$$p(T^*_ZN)=T^*_Zf:=\overline{T^*_{Z_{reg}}f}$$
 the \emph{relative conormal space} of $f|Z$
(i.e. the closure of the fiberwise conormal bundle of $f: Z_{reg}\subset N\to \bC$). 
In particular,  $T^*_Zf$ is irreducible of dimension $n=dim(N)=dim(T^*_ZN)$.
Moreover
$p|T^*_{Z_{reg}}N$ is also injective so that
 $$p_*([T^*_ZN])=[T^*_Zf]\:.$$
Consider the cartesian diagram
\begin{equation}\begin{CD}
i^*T^*N @> i' >> T^*N \\
@V t VV @VV p V \\
T^*M @>> k > T^*_f\:,
\end{CD}\end{equation}
witk $k: T^*M=T^*_f|M\to T^*_f$ the inclusion of the fiber over $\{f=0\}=M$.
Then
$$t(i'^{-1}(T^*_ZN))=k^{-1}(p(T^*_ZN)) \subset T^*M$$
is of pure dimension $n-1=dim(M)$. Moreover, by base change  we get
\begin{equation}\begin{split}
t_*(i'^!(CC(\check{E}u_Z)]) &= t_*(i'^!([T^*_ZN])) \\
&= k^!(p_*([T^*_ZN])) \\
&= k^!([T^*_Zf])\:.
\end{split}\end{equation}

And by \cite{Sab}[Thm.4.3] one has
\begin{equation}\label{sab}
 k^!([T^*_Zf]) =CC(-\psi_f(\check{E}u_Z))\:,
\end{equation}
with $\psi_f: F(Y)\to F(X)$ the \emph{nearby cycles} for constructible functions, which 
are calculated as weighted Euler characteristics of local Milnor fibers
(compare \cite{Sch-lect}[2.4.7]). Then the sign in (\ref{sab}) is again due to the 
use of $(-1)^{dim(S)}$ in the definition (\ref{Morse}), because this local Milnor fiber is transversal
to the strata $S$ of an adapted Whitney stratification, cutting down the complex dimension by one.

Finally we only have to show
$\psi_f(\check{E}u_Z)=i^*(\check{E}u_Z)$. But the difference is just the \emph{vanishing cycles}
(see \cite{Sch-lect}[2.4.8]):
$$\phi_f(\check{E}u_Z):= \psi_f(\check{E}u_Z) - i^*(\check{E}u_Z) \:,$$
with $\phi_f(\check{E}u_Z)(x)=0$ for all $x\in X$ by \cite{BLS}[Thm.3.1]. Or one can use here
\cite{Sch3}[(15)] as an application of the 
\emph{micro-local intersection formula} (see \cite{Sch3}[(13),(14)]):
\begin{equation}
-\phi_f(\check{E}u_Z)(x)= \sharp_{df_x}(df(N)\cap CC(\check{E}u_Z) )\:,
\end{equation}
 since $CC(\check{E}u_Z) = [T^*_ZN]$ and $df(N)\cap T^*_ZN=\emptyset$ by the non-characteristic assumption.\\

Let us finish this paper with another nice application of Theorem \ref{non-cc}.

\begin{cor}\label{cor-int}
Let $M$ a complex (algebraic) manifold of dimension $m=dim(M)$, with $\alpha, \beta\in F(M)$ given constructible functions.
Assume that the diagonal embedding $d: M\to M\times M$ is non-characteristic with respect
to $supp(CC(\alpha\times \beta))$ (e.g. $\alpha, \beta\in F(M)$ are splayed or stratified transversal), with  $supp(\alpha\cdot \beta)$ compact.

Then also  $supp(CC(\alpha)\cap CC(\beta))\subset T^*M$ is compact, with
\begin{equation}\label{mic-int}
\chi(M;\alpha\cdot \beta) = (-1)^m\cdot deg( CC(\alpha)\cap CC(\beta))\:.
\end{equation} 
\end{cor}

\begin{proof} Consider the cartesian diagram
\begin{equation} \begin{CD}
T^*M @> (id,a) >> T^*M\times_M T^*M @> d' >> T^*(M\times M) \\
@V \pi_M VV @VV t V \\
M @>> s > T^*M \:,
\end{CD}\end{equation}
with $s: M\to T^*M$ the zero-section and $a: T^*M\to T^*M$ the antipodal map.
Then one gets  by the \emph{global index formula} for characteristic cycles (see e.g.
\cite{Sch3}[Cor.0.1], which also covers the corresponding context in real geometry):
$$\chi(M;\alpha\cdot \beta) = k_*s^!(CC(\alpha\cdot \beta))\:,$$
with $k: M\to pt$ a constant map. Here in our complex geometric context, this is also a very special case of the functoriality 
of the (dual) Chern MacPherson class with respect to the constant proper map $k: supp(\alpha\cdot \beta)\to pt$,
since 
\begin{equation}
 s^!(CC(\alpha\cdot \beta))=\check{c}_0(\alpha\cdot \beta)=c_0(\alpha\cdot \beta) \in H_0(supp(\alpha\cdot \beta))\:.
\end{equation}
And this equality follows from \cite{Fu}[Example 4.1.8]:
$$\{ c(T^*M)\cap s_*(CC(\alpha\cdot \beta)) \}_0\:,$$
with $\{-\}_0$ the degree zero part, calculates for the $m$-dimensional conic cycle $ CC(\alpha\cdot \beta)$ in the vector bundle $T^*M$ 
of rank $m$ the \emph{intersection with the zero-section}. But by the commutative diagram (\ref{eq:conormal}) we know that
$$c(T^*M)\cap s_*(CC(\alpha\cdot \beta)) =\check{c}_*(\alpha\cdot \beta) \:.$$

By Theorem \ref{non-cc} and the multiplicativity of the characteristic cycle map $CC$  we also have
$$CC(\alpha\cdot \beta)= (-1)^m\cdot t_*d'^!(CC(\alpha)\times CC(\beta))\:.$$
And by the non-characteristic (or splayedness) assumption,
$$t: d'^{-1}(supp(CC(\alpha)\times CC(\beta)))\to T^*M$$
is proper. But then by base change, also $\pi_M$ as a map 
$$(id,a)^{-1}d'^{-1}(supp(CC(\alpha)\times CC(\beta)))=supp(CC(\alpha)\cap a_*CC(\beta))\to M$$
is proper, with image contained in the compact subset $supp(\alpha\cdot \beta)\subset M$.
Note that $a^!=a_*: H_*(T^*M)\to H_*( T^*M)$, since $a^2=id: T^*M\to T^*M$.
Finally by the base change $s^!t_*=\pi_{M*}(id,a)^!$ one gets
\begin{equation*}\begin{split}
\chi(M;\alpha\cdot \beta)  &= (-1)^m\cdot k_*\pi_{M*}(CC(\alpha)\cap a_*(CC(\beta)) \\
&=  (-1)^m\cdot deg( CC(\alpha)\cap a_*CC(\beta))\:.
\end{split}\end{equation*}
And in our complex context we also have $ a_*CC(\beta)=CC(\beta)$.
\end{proof}

\begin{rem}
In \cite{GrM} a counterpart of Corollary \ref{cor-int}  in the context of real geometry
is discussed under a ``stratified tranversality'' assumption. See also \cite{ScSc}[(1.4), Part II] for a far reaching generalization to
\emph{elliptic pairs}.  Note that both references use a different orientation convention, so that the sign
$(-1)^m$ in (\ref{mic-int}) disappears.
\end{rem}

% ------------------------------------------------------------------------

\subsection*{Acknowledgment}
This paper is an extended version of a talk given at a conference in Merida (Mexico 2014)
for the celebration of the 60th birthday of Pepe Seade. Here I would like to thank the organizers for the invitation to this wonderful conference.
It is a pleasure to thank Pepe Seade for so many discussions over the years on the theory of Chern classes of singular spaces.
The author is also  greatful to Paolo Aluffi und Eleonore Faber for their inspiring papers as well as 
for some communications regarding the subject of this paper. The author was supported by the SFB 878 groups, geometry and actions. 

% ------------------------------------------------------------------------


\begin{thebibliography}{11}

\bibitem{AB}   
P. Aluffi, and J.P. Brasselet: \textit{Une nouvelle preuve de la concordance des classes d\'efinies par M.H. Schwartz et par R. MacPherson}.
 Bull. Soc. Math. France  \textbf{13}6  (2008),  159--166. 

\bibitem{AF}   
P. Aluffi, and E. Faber: \textit{Splayed divisors and their Chern classes}.
J. Lond. Math. Soc.  \textbf{88} (2013), 563-–579.  

\bibitem{AF2}   
P. Aluffi, and E. Faber: \textit{Chern classes of splayed intersections}.
Canadian Journal of Math. \textbf{67} (2015), 1201--1218.

\bibitem{B}   K. Behrend: \textit{Donaldson-Thomas type invariants via microlocal geometry}.
 Ann. of Math.  \textbf{170}  (2009), 1307--1338.

\bibitem{Br} 
J.P. Brasselet: \textit{  Existence des classes de Chern en th\'eorie bivariante}.
 Ast\'erisque, \textbf{101--102} (1981), 7--22.


\bibitem{BLS} 
J.P: Brasselet,  L\^e D\~ung Tr\'ang and J. Seade: \textit{ Euler obstruction and indices of vector fields}.
Topology \textbf{39} (2000), 1193--1208.


\bibitem{BSY1} 
J.P. Brasselet, J. Sch\"urmann and S. Yokura:  \textit{ Hirzebruch classes and motivic Chern classes for singular spaces}.
  Journal of Topology and Analysis, \textbf{2},  (2010), 1--55.


\bibitem{BSY2} 
J.P. Brasselet, J. Sch\"urmann and S. Yokura:  \textit{ On the uniqueness of bivariant Chern classes and bivariant Riemann--Roch transformations}.  Advances in Math., \textbf{210} (2007), 797--812.

%\bibitem{BSY3} 
%J.P. Brasselet, J. Sch\"urmann and S. Yokura: \textit{On Grothendieck transformation in Fulton--MacPherson's bivariant theory}.
%  J. Pure and Applied Algebra, \textbf{211} (2007), 665--684.



\bibitem{BS}
J.P. Brasselet and M.-H. Schwartz:
\textit{  Sur les classes de Chern d'une ensemble analytique complexe}.
Ast\'erisque \textbf{ 82--83}
(1981), 93--148.

\bibitem{CMS}  R. Callejas-Bedregal, M.F.Z. Morgado and J. Seade: \textit{ On the Milnor classes of local complete intersections}.
arXiv:1208.5084

\bibitem{Fa} E. Faber: \textit{ Towards transversality of singular varieties: splayed divisors}.
 Publ. Res. Inst. Math. Sci.  \textbf{49}  (2013),  393--412.

\bibitem{FuMC}
J. Fu, and C. McCrory: \textit{ Stiefel-Whitney classes and the conormal cycle of a singular variety}.
  Trans. Amer. Math. Soc. \textbf{349} (1997), 809--835.

\bibitem{Fu}
W.  Fulton:  \textit{ Intersection theory},
   Springer Verlag, 1984.

\bibitem{FM}
W. Fulton and R. MacPherson: ,\textit{Categorical frameworks for the study of singular spaces}.
Memoirs of Amer. Math. Soc. \textbf{ 243}, 1981.

\bibitem{Gi1}
V. Ginsburg:  \textit{ Characteristic cycles and vanishing cycles}.
  Inv. Math. \textbf{84} (1986), 327--402. 

\bibitem{Gi2}
 V. Ginsburg: \textit{  $g$-Modules, Springer's representations and bivariant  Chern classes}.
  Adv. in Math. \textbf{61} (1986), 1--48. 

\bibitem{GM}
 M.  Goresky and R. MacPherson,  \textit{ Intersection homology theory}.
  Topology \textbf{149} (1980), 155--162.



\bibitem{GrM}
M. Grinberg, and R. MacPherson: \textit{  Euler characteristics and Lagrangian intersections}.
 In: Symplectic geometry and topology, IAS/Park City Math. Ser. \textbf{7},  Amer. Math. Soc. Providence 
 (1999), 265--293.


\bibitem{KS}
M.  Kashiwara, and P.  Schapira: \textit{  Sheaves on Manifolds}.
  Springer, Berlin Heidelberg, 1990.

\bibitem{Ken1}
G.  Kennedy:  \textit{ MacPherson's Chern classes of singular varieties}.
  Com. Algebra. \textbf{9}  (1990), 2821--2839.

\bibitem{K}
M.   Kwiecinski:  \textit{ Formule de produit pour les classes caract\'{e}ristiques
   de Chern-Schwartz-MacPherson et homologie d'intersection}.
   C.R. Acad. Sci. Paris \textbf{314} (1992), 625--628.

\bibitem{KYo}    
M. Kwiecinski, and S. Yokura:  \textit{ Product formula for twisted MacPherson Classes}.
   Proc. Japan Acad. \textbf{68}  (1992), 167--171.

\bibitem{MP} R. MacPherson, {\textit  Chern classes for singular varieties}.
Ann. of Math. \textbf{100} (1974), 423--432.

\bibitem{Matsui}
A. Matsui:  \textit{ Intersection formula for Stiefel-Whitney homology classes}.
   Tohoku Math. J. \textbf{40}  (1988), 315--322.

\bibitem{Matsui2}
A. Matsui:  \textit{ Hirzebruch L-homology classes and the intersection formula}.
   Kodai Math. J. \textbf{12} (1989), 56--71 

\bibitem{Sab}
C.  Sabbah: \textit{ Quelques remarques sur la g\'{e}om\'{e}trie des espaces conormaux}.
 Ast\'erisque  \textbf{130}  (1985), 161--192.
 
\bibitem{ScSc}
P. Schapira and J.P. Schneiders: \textit{ Index theorems for elliptic pairs}.
Ast\'{e}risque \textbf{224} (1994). 

\bibitem{Sch4}
J. Sch\"{u}rmann, \textit{ A generalized Verdier-type Riemann-Roch theorem for Chern-Schwartz-MacPherson classes}.
math.AG/0202175

\bibitem{Sch-book}
J. Sch\"urmann, \textit{ Topology of singular spaces and constructible sheaves}.
  Mathematics Institute of the Polish Academy of Sciences, 
Mathematical Monographs (New Series), \textbf{63}, Birkh\" auser Verlag, Basel, 2003.

\bibitem{Sch3}
J. Sch\"urmann, \textit{ A general intersction formula for Lagrangian cycles}.
Comp. Math. \textbf{140} (2004), 1037--1052.


\bibitem{Sch-lect}
J. Sch\"{u}rmann, \textit{ Lectures on characteristic classes of constructible functions}.
 Trends Math.:
Topics in cohomological studies of algebraic varieties (Ed. P. Pragacz), 175--201,
Birkh\"{a}user, Basel, 2005.


\bibitem{Sch-MSRI}
J. Sch\"urmann: \textit{Characteristic classes of mixed Hodge modules}. 
 in ``Topology of Stratified Spaces", MSRI Publications  Vol. \textbf{58}, Cambridge University Press (2010) , 419--470.

\bibitem{ST}
J. Sch\"{u}rmann, and M. Tib\u ar:  \textit{ Index formula for MacPherson cycles of affine algebraic varieties}.
Tohoku Math. J. \textbf{62} (2010),  29-–44.


\bibitem{Schw1}
M.H. Schwartz: 
\textit{ Classes caract{\'e}ristiques d{\'e}finies par une
stratification d'une vari{\'e}t{\'e}s analytique complexe}.
C. R. Acad. Sci.Paris \textbf{t. 260} 
(1965), 3262--3264, 3535--3537.

\bibitem{Schwa2}
M.H. Schwartz:
\textit{ Classes et caract\`eres de Chern des espaces lin\'eaires}.
Pub. Int. Univ. Lille, \textit{ 2 Fasc. 3}
(1980) and C. R. Acad. Sci. Paris S\'er. I \textbf{295} (1982), 399--402.

\bibitem{Yo4}
S. Yokura: \textit{  On a Verdier-type Riemann-Roch for Chern-Schwartz-MacPherson class}.
   Topology and its Appl. \textbf{94} (1999), 315--327. 


  
   
\end{thebibliography}
\end{document}